\documentclass[12pt]{amsart}
\usepackage[utf8]{inputenc}
\usepackage{amsfonts}
\usepackage{amsmath}
\usepackage{amssymb}
\usepackage{amsthm}
\usepackage{xcolor}
\usepackage[margin=1.3in]{geometry}
\usepackage{graphicx}
\usepackage[11pt]{moresize}
\usepackage{tikz}
\tikzstyle{vertex}=[circle,draw=black,fill=black,inner sep=0,minimum size=3pt,text=white,font=\footnotesize]
\usepackage{nccmath}

\newtheorem{theorem}{Theorem}
\newtheorem*{conjecture*}{Conjecture}
\newtheorem{proposition}{Proposition}[section]
\newtheorem{lemma}[proposition]{Lemma}
\newtheorem{corollary}[proposition]{Corollary}
\newtheorem{claim}[proposition]{Claim}
\theoremstyle{remark}

\newtheorem*{remark*}{Remark}

\newcommand{\vs}{\vspace{3mm}}

\newcommand{\C}{\mathbb{C}}
\newcommand{\D}{\mathbb{D}}
\newcommand{\E}{\mathbb{E}}

\newcommand{\mc}{\mathcal}
\newcommand{\Prob}{\mathbb{P}}

\newcommand{\ep}{\epsilon}
\newcommand{\lam}{\lambda}
\newcommand{\sub}{\subseteq}

\newcommand{\ol}{\overline}
\newcommand{\wt}{\widetilde}

\DeclareMathOperator{\Real}{Re}
\DeclareMathOperator{\Imag}{Im}

\title{New Upper Bounds for Trace Reconstruction}
\author{Zachary Chase}
\thanks{The author is partially supported by Ben Green's Simons Investigator Grant 376201 and gratefully acknowledges the support of the Simons Foundation.}
\address{Mathematical Institute, Andrew Wiles Building, Radcliffe Observatory Quarter, Woodstock Road, Oxford OX2 6GG, UK}
\email{zachary.chase@maths.ox.ac.uk}
\date{September 7, 2020}

\begin{document}

\begin{abstract}
We show that any $n$-bit string can be recovered with high probability from $\exp(\widetilde{O}(n^{1/5}))$ independent random subsequences. 
\end{abstract}

\maketitle

\section{Introduction}

Given a string $x \in \{0,1\}^n$, a \textit{trace} of $x$ is a random string obtained by deleting each bit of $x$ with probability $q$, independently, and concatenating the remaining string. For example, a trace of $11001$ could be $101$, obtained by deleting the second and third bits. The goal of the trace reconstruction problem is to determine an unknown string $x$, with high probability, by looking at as few independently generated traces of $x$ as possible. 

\vs

More precisely, fix $\delta,q \in (0,1)$. Take $n$ large. For each $x \in \{0,1\}^n$, let $\mu_x$ be the probability distribution on $\cup_{j=0}^n \{0,1\}^j$ given by $\mu_x(w) = (1-q)^{|w|}q^{n-|w|}f(w;x)$, where $f(w;x)$ is the number of times $w$ appears as a subsequence in $x$, that is, the number of strictly increasing tuples $(i_0,\dots,i_{|w|-1})$ such that $x_{i_j} = w_j$ for $0 \le j \le |w|-1$. The problem is to determine the minimum value of $T = T_{q,\delta}(n)$ for which there exists a function $F: (\cup_{j=0}^n \{0,1\}^j)^T \to \{0,1\}^n$ satisfying $\Prob_{\mu_x^T}[F(U^1,\dots,U^T) = x] \ge 1-\delta$ for each $x \in \{0,1\}^n$ (where the $U^j$ denote the $T$ independent traces).

\vs

Supressing the dependence on $q$ and $\delta$, Holenstein, Mitzenmacher, Panigrahy, and Wieder \cite{hmpw} established an upper bound, that $\exp(\wt{O}(n^{1/2}))$ traces suffice. Nazarov and Peres \cite{nazarovperes} and De, O'Donnell, and Servedio \cite{dds} simultaneously obtained the (previous) best upper bound known, that $\exp(O(n^{1/3}))$ traces suffice. 

\vs

In this paper, we improve the upper bound on trace reconstruction to $\exp(\wt{O}(n^{1/5}))$.

\vspace{1mm}

\begin{theorem}\label{main}
For any deletion probability $q \in (0,1)$ and any $\delta > 0$, there exists $C > 0$ so that any unkown string $x \in \{0,1\}^n$ can be reconstructed with probability at least $1-\delta$ from $T = \exp(Cn^{1/5}\log^5 n)$ i.i.d. traces of $x$.
\end{theorem}

\vspace{1mm}

Batu et al. \cite{bkkm} proved a lower bound of $\Omega(n)$, which was improved to $\wt{\Omega}(n^{5/4})$ by Holden and Lyons \cite{holdenlyons}, which was then improved to $\wt{\Omega}(n^{3/2})$ by the author \cite{chasetr}. 

\vs

A variant of the trace reconstruction problem requires one to, instead of reconstruct any string $x$ from traces of it, reconstruct a string $x$ chosen uniformly at random from traces of it. For a formal statement of the problem, see Section 1.2 of \cite{holdenlyons}. Peres and Zhai \cite{pereszhai} obtained an upper bound of $\exp(O(\log^{1/2} n))$ for $q < \frac{1}{2}$, which was then improved to $\exp(O(\log^{1/3} n))$ for all (constant) $q$ by Holden, Pemantle, Peres, and Zhai \cite{hppz}. 

\vs

Holden and Lyons \cite{holdenlyons} proved a lower bound for this random variant of $\wt{\Omega}(\log^{9/4}n)$, which was then improved by the author \cite{chasetr} to $\wt{\Omega}(\log^{5/2}n)$. 

\vs

Several other variants of the trace reconstruction problem have been considered. The interested reader should refer to \cite{bcfrs}, \cite{bcss}, \cite{drr}, \cite{cgor}, \cite{bls}, \cite{narayanan}, \cite{kmmp}, \cite{circulartr}.

\vs

In a previous version of this paper, we proved Theorem \ref{main} only for $q \in (0,\frac{1}{2}]$. Shyam Narayanan found a short argument extending our methods to get all $q \in (0,1)$. He kindly allowed us to use his argument in this paper. 

\vs

We made no effort to optimize the (power of the) logarithmic term $\log^5n$ in Theorem \ref{main}.

\vspace{1mm}

\section{Notation}

We index starting at $0$. For strings $w$ and $x$, we sometimes write $1_{x_{k+i} = w_i}$ as shorthand for $\prod_{i=0}^{|w|-1} 1_{x_{k+i} = w_i}$. Let $\mathbb{D} = \{z \in \C : |z| < 1\}$. For functions $f$ and $g$, we say $f = \wt{O}(g)$ if $|f| \le C|g|\log^C|g|$ for some constant $C$. The symbol $\E_x$ denotes the expectation under the probability distribution over traces generated by the string $x$. For a trace $U$, we define $U_j = 2$ for $j > |U|$; this is simply to make ``$U_j = 0$" and ``$U_j = 1$" both false. We use $0^0 := 1$. For a positive integer $n$, denote $[n] := \{1,\dots,n\}$. For a function $f$ and a set $E$, denote $||f||_E := \max_{z \in E} |f(z)|$. We say $A \sub \{0,\dots,n-1\}$ is \textit{$d$-separated} if distinct $a,a' \in A$ have $|a-a'| \ge d$.

\vspace{1mm}

\section{Sketch of Argument}

The upper bound of $\exp(O(n^{1/3}))$ was obtained by analyzing the polynomial $\sum_k [x_k-y_k]z^k$ whose value can be well enough approximated from a sufficient number of traces. In this paper, we analyze the polynomial $\sum_k [1_{x_{k+i} = w_i}-1_{y_{k+i} = w_i}]z^k$, for a well-chosen (sub)string $w$; its value can be well enough approximated from a sufficient number of traces, provided $q \le 1/2$. The benefit of this polynomial is that for certain choices of $w$, it is far sparser than the more general $\sum_k [x_k-y_k]z^k$. In the author's paper \cite{chasesw} improving the upper bound on the separating words problem, lower bounds were obtained for (the absolute value of) these sparser polynomials near $1$ on the real axis that were superior to those for the more general $\sum_k [x_k-y_k]z^k$. We use the methods developed in that paper and methods used in \cite{littlewoodcircle} to obtain superior lower bounds for points on a small arc of the unit circle centered at $1$.

\vspace{1.5mm}

\section{Proof of Theorem \ref{main}}

Fix $q \in (0,1)$, and let $p = 1-q$. The following `single bit statistics' identity was proven in \cite[Lemma 2.1]{nazarovperes}; in it, $U$ denotes a random trace of $x$. $$\E_x\left[p^{-1}\sum_{0 \le j \le n-1} 1_{U_j=1}\left(\frac{z-q}{p}\right)^j\right] = \sum_{0 \le k \le n-1} 1_{x_k=1} z^k.$$ We shall use a generalization of this identity to approximate a weighted count (by position) of subsequence appearances in $x$ rather than a weighted count (by position) of appearances of $1$. Choosing variables appropriately will recover a weighted count of \textit{(contiguous) substring} appearances in $x$. An unweighted version was used in \cite{cdlss}. 

\vspace{1mm}

\begin{proposition}\label{generalidentity}
For any $x \in \{0,1\}^n, l \ge 1, w \in \{0,1\}^l$, and $z_0,\dots,z_{l-1} \in \C$, we have $$\E_x\left[p^{-1} \sum_{j_0 < \dots < j_{l-1}} \left(\prod_{i=0}^{l-1} 1_{U_{j_i} = w_i}\right)\left(\frac{z_0-q}{p}\right)^{j_0}\left(\prod_{i=1}^{l-1} \left(\frac{z_i-q}{p}\right)^{j_i-j_{i-1}-1}\right)\right]$$ $$ = \sum_{k_0 < \dots < k_{l-1}} \left(\prod_{i=0}^{l-1} 1_{x_{k_i}=w_i}\right) z_0^{k_0}\left(\prod_{i=1}^{l-1} z_i^{k_i-k_{i-1}-1}\right).$$
\end{proposition}

\begin{proof}
To ease with the proof and perhaps give the reader another perspective by ``writing out the products", we rewrite the identity we wish to prove as $$\E_x\left[p^{-l}\sum_{\substack{0 \le j \le n-1 \\ \Delta_1,\dots,\Delta_{l-1} \ge 1}} 1_{\wt{U}_j = w_0} 1_{\substack{\wt{U}_{j+\Delta_1+\dots+\Delta_i} = w_i \\ \forall 1 \le i \le l-1}} (\frac{z_0-q}{p})^j(\frac{z_1-q}{p})^{\Delta_1-1}(\frac{z_2-q}{p})^{\Delta_2-1}\dots (\frac{z_{l-1}-q}{p})^{\Delta_{l-1}-1}\right]$$ $$= \sum_{k_0 < \dots < k_{l-1}} 1_{x_{k_0} = w_0,\dots, x_{k_{l-1}} = w_{l-1}} z_0^{k_0}z_1^{k_1-k_0-1}z_2^{k_2-k_1-1}\dots z_{l-1}^{k_{l-1}-k_{l-2}-1}.$$ By basic combinatorics, the left hand side of the above is $$p^{-l}\sum_{j,\Delta_1,\dots,\Delta_{l-1}} \sum_{k_0 < \dots < k_{l-1}} 1_{\substack{x_{k_i} = w_i \\ \forall 0 \le i \le l-1}} {k_0 \choose j}{k_1-k_0-1 \choose \Delta_1-1}{k_2-k_1-1 \choose \Delta_2-1}\dots {k_{l-1}-k_{l-2}-1 \choose \Delta_{l-1}-1}$$ $$\hspace{45mm} \times p^{j+\Delta_1+\dots+\Delta_{l-1}+1}q^{k_{l-1}+1-(j+\Delta_1+\dots+\Delta_{l-1}+1)}$$ $$\hspace{49.3mm} \times (\frac{z_0-q}{p})^j(\frac{z_1-q}{p})^{\Delta_1-1}\dots (\frac{z_{l-1}-q}{p})^{\Delta_{l-1}-1}$$ $$ = \sum_{k_0 < \dots < k_{l-1}} 1_{\substack{x_{k_i} = w_i \\ \forall 0 \le i \le l-1}} \left(\sum_j {k_0 \choose j}(z_0-q)^j q^{k_0-j}\right)\left(\sum_{\Delta_1} {k_1-k_0-1 \choose \Delta_1-1} (z_1-q)^{\Delta_1-1}q^{k_1-k_0-1-(\Delta_1-1)}\right)$$ $$ \times \dots \times \left(\sum_{\Delta_{l-1}} {k_{l-1}-k_{l-2}-1 \choose \Delta_{l-1}-1} (z_{l-1}-q)^{\Delta_{l-1}-1}q^{k_{l-1}-k_{l-2}-1-(\Delta_{l-1}-1)}\right).$$ The binomial theorem finishes the proof. 
\end{proof}

\vs

Let $\mc{P}_n$ be the set of all polynomials\footnote{Throughout the paper, we omit floor functions when they don't meaningfully affect anything.} $p(z) = 1-\sigma z^d+\sum_{j = n^{1/5}}^n c_j z^j \in \C[z]$ with $1 \le d < n^{1/5}, \sigma \in \{0,1\}$, and $|c_j| \le 1$ for each $j$.

\vs

We prove the following theorem in the next section. We assume it to be true until then.

\vspace{1.5mm}

\begin{theorem}\label{arclowerbound}
There is some $C > 0$ so that for any $n \ge 2$ and any $p \in \mc{P}_n$, $$\max_{|\theta| \le n^{-2/5}} |p(e^{i \theta})| \ge \exp(-Cn^{1/5}\log^5n).$$
\end{theorem}

\vspace{1mm}

\begin{proposition}\label{choosez0}
For any distinct $x,y \in \{0,1\}^n$ with $x_i = y_i$ for all $0 \le i < 2n^{1/5}-1$, there are $w \in \{0,1\}^{2n^{1/5}}$ and $z_0 \in \{e^{i\theta} : |\theta| \le n^{-2/5}\}$ such that $$\left| \sum_k [1_{x_{k+i}=w_i}-1_{y_{k+i}=w_i}]z_0^k \right| \ge \exp(-Cn^{1/5}\log^5 n).$$
\end{proposition}

\begin{proof}
Let $i \ge 2n^{1/5}-1$ be the first index with $x_i \not = y_i$. Let $w' = x_{i-2n^{1/5}+1},\dots,x_{i-1}$. As used in \cite{chasesw}, Lemmas $1$ and $2$ of \cite{robson} imply that there is some choice $w \in \{w'0,w'1\}$ such that the indices $k$ for which $x_{k+i} = w_i$ for all $0 \le i \le 2n^{1/5}-1$ are $n^{1/5}$-separated, and such that the indices $k$ for which $y_{k+i} = w_i$ for all $0 \le i \le 2n^{1/5}-1$ are $n^{1/5}$-separated. Therefore, if $p(z) := \sum_k [1_{x_{k+i} = w_i}-1_{y_{k+i} = w_i}]z^k$, then $\epsilon \frac{p(z)}{z^m} \in \mc{P}_n$ for some $\epsilon \in \{-1,1\}$ and $0 \le m \le n-1$. Thus, by Theorem \ref{arclowerbound}, there is some $\theta \in [-n^{-2/5},n^{-2/5}]$ such that $\exp(-Cn^{1/5}\log^5 n) \le |\epsilon \frac{p(e^{i\theta})}{e^{ i m\theta}}| = |p(e^{ i \theta})|$. Take $z_0 = e^{i\theta}$. 
\end{proof}

\vs

In a previous version of this paper, we used Proposition \ref{generalidentity} with $z_1,\dots,z_{l-1} = 0$ and $z_0$ chosen according to Proposition \ref{choosez0} to prove Theorem \ref{main}, which only worked for $q \le 1/2$, since, for $q > 1/2$, the quantity $(-q/p)^{j_i-j_{i-1}}$ would be too large in magnitude (for $j_i-j_{i-1} \approx n$), leading to too large a variance to well-enough approximate $\sum_k [1_{x_{k+i}=w_i}-1_{y_{k+i}=w_i}]z_0^k$ with few traces. The idea of Shyam Narayanan was to choose $z_1,\dots,z_{l-1}$ close to $1$ so that $(\frac{z_i-q}{p})^{j_i-j_{i-1}}$ would no longer be too large in magnitude, while also keeping the right hand side of Proposition \ref{generalidentity} not too small. The following corollary, due to him, establishes the existence of such $z_1,\dots,z_{l-1}$. 

\vspace{1mm}

\begin{corollary}\label{choosez0z1}
For any distinct $x,y \in \{0,1\}^n$ with $x_i = y_i$ for all $0 \le i < l-1 := 2n^{1/5}-1$, there are $w \in \{0,1\}^l, z_0 \in \{e^{i\theta} : |\theta| \le n^{-2/5}\}$, and $z_1,\dots,z_{l-1} \in [1-2p,1]$ such that\footnote{We similarly abuse notation by writing $1_{x_{k_i} = w_i}$ to denote $\prod_{i=0}^{l-1} 1_{x_{k_i}=w_i}$.} $$\left|\sum_{k_0 < \dots < k_{l-1}} [1_{x_{k_i} = w_i}-1_{y_{k_i}=w_i}] z_0^{k_0}z_1^{k_1-k_0-1}\dots z_{l-1}^{k_{l-1}-k_{l-2}-1}\right| \ge \exp(-C'n^{1/5}\log^5 n).$$
\end{corollary}

\begin{proof}
Let $w$ and $z_0$ be those guaranteed by Proposition \ref{choosez0}. Let $$f(z_1) = {n \choose 2n^{1/5}}^{-1}\sum_{k_0 < \dots < k_{l-1}} [1_{x_{k_i}=w_i}-1_{y_{k_i}=w_i}] z_0^{k_0}z_1^{k_{l-1}-k_0-(l-1)}.$$ Note that $f$ is a polynomial in $z_1$ with each coefficient trivially upper bounded by $1$ in absolute value. Therefore, by Theorem 5.1 of \cite{littlewood01}, \begin{align*} {n \choose 2n^{1/5}}\max_{z_1 \in [1-2p,1]} |f(z_1)| &\ge {n \choose 2n^{1/5}} |f(0)|^{c_1/(2p)}e^{-c_2/(2p)} \\ &\ge {n \choose 2n^{1/5}} \left({n \choose 2n^{1/5}}^{-1}\exp(-Cn^{1/5}\log^5 n)\right)^{c_1/(2p)}e^{-c_2/(2p)} \\ &\ge \exp(-C'n^{1/5}\log^5 n).\end{align*} The corollary then follows by taking a $z_1$ realizing this maximum and then setting $z_2,\dots,z_{l-1}=z_1$.
\end{proof}

\vspace{1.5mm}

We are now ready to establish our main theorem. We encourage the reader to first read the proof of the $\exp(O(n^{1/3}))$ upper bound in \cite{nazarovperes}. 

\vspace{1mm}

\begin{proof}[Proof of Theorem \ref{main}]
Take distinct $x,y \in \{0,1\}^n$. If $x_i \not = y_i$ for some $i < 2n^{1/5}-1$, then, by Lemma 4.1 of \cite{pereszhai}, $x$ and $y$ can be distinguished with high probability with $\exp(O(n^{1/15})) \le \exp(C''n^{1/5}\log^5 n)$ traces\footnote{Alternatively, one may simply ``make life harder" by adding enough $0$s, say, to the start of $x$ and $y$.}. So suppose otherwise. Let $w,z_0,z_1,\dots,z_{2n^{1/5}-1}$ be those guaranteed by Corollary \ref{choosez0z1}. Since $z_1,\dots,z_{2n^{1/5}-1} \in [1-2p,1]$, each of $\frac{z_i-q}{p}$, $1 \le i \le 2n^{1/5}-1$, is between $-1$ and $1$, and so the expression in brackets in Proposition \ref{generalidentity} has magnitude upper bounded by $n|\frac{z_0-q}{p}|^n 2^{2n^{1/5}}$, which, by the choice of $z_0$, is upper bounded by $n\exp(C\frac{n}{n^{4/5}})2^{2n^{1/5}}$ (see \cite[(2.3)]{nazarovperes} for details). Therefore, since the expression in brackets in Proposition \ref{generalidentity} is a function of just the observed traces, by Corollary \ref{choosez0z1} and a standard Höeffding inequality argument (see \cite{nazarovperes} for details; note the pigeonhole is not necessary), we see that $\exp(C'''n^{1/5}\log^5 n)$ traces suffice to distinguish between $x$ and $y$. As explained in \cite{nazarovperes}, this ``pairwise upper bound" in fact suffices to establish Theorem \ref{main}. 
\end{proof}

\vspace{2mm}

\section{Proof of Theorem \ref{arclowerbound}}

\vspace{1mm}

We may of course assume $n$ is large.

\vs

Let $a = n^{-2/5}$ and $r = a^{-1/2}$. Let $r_* \in [r]$ be such that $$\sum_{j=1}^{r_*} \frac{1}{\log^2(j+3)}-\sum_{j=r_*+1}^r \frac{1}{\log^2(j+3)} \in [20,21];$$ such an $r_*$ clearly exists. Let $$\begin{cases} \epsilon_j = +1 & \text{if }  1 \le j \le r_* \\ \epsilon_j = -1 & \text{if } r_*+1 \le j \le r\end{cases}.$$ Let $\lambda_a \in (1,2)$ be such that $$\sum_{j=1}^r \frac{\lambda_a}{j^2\log^2(j+3)} = 1.$$ Let $$d_j = \frac{\lambda_a}{j^2\log^2(j+3)}.$$ Define $$\wt{h}(z) = \wt{\lambda}_a\sum_{j=1}^r \epsilon_j d_j z^j,$$ where $\wt{\lambda}_a \in (1,2)$ is such that $\wt{h}(1) = 1$. Define $$h(z) = (1-a^{10})\wt{h}(z).$$ Let $$\alpha = e^{ia}, \beta = e^{-ia},$$ and $$I_t = \{z \in \C : \arg(\frac{\alpha-z}{z-\beta}) = t\}$$ for $t \ge 0$. Note that $I_0$ is the line segment connecting $\alpha$ and $\beta$ and $I_a = \{e^{i \theta} : |\theta| \le a\}$ is the set on which we wish to lower bound $p$ at some point. Let $$G_a = \{z \in \C : \arg(\frac{\alpha-z}{z-\beta}) \in (\frac{a}{2},a)\}$$ be the open region bounded by $I_{a/2}$ and $I_a$. 

\vs

As in \cite{chasesw}, we needed our choice of $h$ to satisfy (i) $|h(e^{2\pi i t})| \le 1-c|t|$ for $|t| > a^{1/2}$ (up to logs). In this paper, we need (ii) $|h(e^{2\pi i t})| \ge 1-Ca^2$ for $|t| \approx a$; in \cite{chasesw}, we instead had $|h(e^{2\pi i t})| \approx 1-a$ for $|t| \approx a$. Some thought shows that a polynomial with positive coefficients will not work. We therefore had roughly half of our coefficients be $-1$ so that (ii) holds; changing those coefficients doesn't affect (i) since the corresponding degrees are large. However, due to our required normalization that $h(1)$ is basically $1$, the negative coefficients make it so that $h$ might no longer map into the unit disk, which is highly problematic for later application. Luckily, though, $\wt{h}$, and thus $h$, \textit{does} map into the unit disk. We prove that in the appendix.

\begin{lemma}\label{unitdisk}
For any $t \in [-\pi,\pi]$, $\wt{h}(e^{it}) \in \ol{\D}$. 
\end{lemma}

\vspace{0.5mm}

\begin{lemma}\label{hproperties}
There are absolute constants $c_4,c_5,C_6 > 0$ such that the following hold for $a > 0$ small enough. First, $h(e^{2\pi i t}) \in G_a$ for $|t| \le c_4a$. Second, $|h(e^{2\pi i t})| \le 1-c_5\frac{|t|}{\log^2(a^{-1})}$ for $t \in [\frac{-1}{2},\frac{1}{2}]\setminus [-C_6a^{1/2},C_6a^{1/2}]$.
\end{lemma} 

\begin{proof}
Take $|t| \le a$. Then, \begin{align*} \wt{h}(e^{2\pi i t}) &= \wt{\lam}_a\sum_{j=1}^{r_*} \frac{\lambda_a}{j^2\log^2(j+3)}(1+2\pi i tj-2\pi^2t^2j^2+O(t^3j^3)) \\ &\hspace{15mm} -\wt{\lam}_a\sum_{j=r_*+1}^r \frac{\lambda_a}{j^2\log^2(j+3)}(1+2\pi i tj-2\pi^2t^2j^2+O(t^3j^3)).\end{align*} By our choice of $r_*$, $h(e^{2\pi i t}) = 1-\delta+\epsilon i$ for $\delta := c_1 t^2+a^{10}+O(\frac{t^3 r^2}{\log^2 r})$ and $\ep := c_2 t+O(\frac{t^3 r^2}{\log^2 r})$, where $c_1,c_2$ are bounded positive quantities that are bounded away from $0$. By multiplying the denominator by its conjugate, we have $$\arg\left(\frac{e^{ia}-(1-\delta+\epsilon i)}{(1-\delta+\epsilon i)-e^{-ia}}\right) = \arg\Big(\left[e^{ia}-(1-\delta+\epsilon i)\right]\cdot\left[(1-\delta-\epsilon i)-e^{ia}\right]\Big).$$ The ratio of the imaginary part to the real part of the term inside $\arg(\cdot)$ is $$\frac{2(1-\delta-\cos(a))\sin(a)}{-\cos^2(a)+2(1-\delta)\cos(a)-(1-\delta)^2+\sin^2(a)-\epsilon^2}.$$ Writing $\cos(a) = 1-\frac{1}{2}a^2+O(a^4)$ and $\sin(a) = a+O(a^3)$, and using $\delta = O(a^2)$, the above simplifies to $$\frac{a^3-2a\delta+O(a^4)}{a^2-\epsilon^2+O(a^3)}.$$ If $|t| \le c_4 a$, then, as $\delta = c_1t^2+a^{10}+O(\frac{t^3 r^2}{\log^2 r}), \ep = c_2 t+O(\frac{t^3 r^2}{\log^2 r})$, the inverse tangent of the above is at least $\frac{a}{2}$; the arctangent is at most $a$, since, by Lemma \ref{unitdisk}, $h(e^{2\pi i t})$ lies in the unit disk (alternatively, one may note $2a\delta > \epsilon^2$). 

\vs

We now establish the second part of the lemma. What \cite{chasesw} shows is $$\left|\sum_{j=1}^m \frac{\lambda_a e^{2\pi i tj}}{j^2\log^2(j+3)}\right| \le 1-\frac{\lambda_a |t|}{3\log^2(m+3)}+\frac{\lambda_a}{m\log^2(m+3)}$$ for any $m \ge 1$ and $t \in [-\frac{1}{2},\frac{1}{2}]\setminus[-3m^{-1},3m^{-1}]$. For $m = r_*$, if $|t| > C_6 a^{1/2}$, for say $C_6 = 100$, then certainly $3|t|^{-1} < m$, and so we have \begin{equation}\label{sumbypartsbound} \left|\sum_{j=1}^{r_*} \frac{\lam_a e^{2\pi i tj}}{j^2\log^2(j+3)}\right| \le 1-c\frac{|t|}{\log^2(a^{-1})}.\end{equation} We can crudely bound \begin{equation}\label{crudebound} \left|\sum_{j=r_*+1}^r \frac{\lambda_a e^{2\pi i tj}}{j^2\log^2(j+3)}\right| \le \frac{4}{\log^2(a^{-1})}\frac{1}{r_*}.\end{equation} Combining \eqref{sumbypartsbound} and \eqref{crudebound}, we obtain $$\left|\sum_{j=1}^r \frac{\lam_a \ep_j e^{2\pi i tj}}{j^2\log^2(j+3)}\right| \le 1-c_5'\frac{|t|}{\log^2(a^{-1})}$$ for $|t| \ge C_6 r^{-1}$, with $c_5' > 0$ small and $C_6$ large enough. Now, since \begin{align*} \wt{\lam}_a^{-1} &= \sum_{j=1}^{r_*} \frac{\lam_a}{j^2\log^2(j+3)}-\sum_{j=r_*+1}^r \frac{\lam_a}{j^2\log^2(j+3)} \\ &= 1-2\sum_{j=r_*+1}^r \frac{\lam_a}{j^2\log^2(j+3)} \\ &\ge 1-2\frac{2}{\log^2(a^{-1})}\frac{2}{r_*} \\ &\ge 1-\frac{20}{r\log^2 (a^{-1})},\end{align*} we see $$\left|\wt{\lam}_a\sum_{j=1}^r \frac{\lam_a \ep_j e^{2\pi i tj}}{j^2\log^2(j+3)} \right| \le 1-c_5\frac{|t|}{\log^2(a^{-1})}$$ for $|t| \ge C_6 r^{-1}$, provided $C_6$ is large enough. Since $1-a^{10} \le 1$, we are done.
\end{proof}

\vs

Let $m = c_4^{-1}n^{2/5}, J_1 = c_5^{-1}n^{-1/5}m\log^4 n$, and $J_2 = m-J_1$. A minor adapation of the relevant proof in \cite{chasesw} proves the following. 

\vspace{1mm}

\begin{lemma}\label{tildepproduct}
Suppose $\wt{p}(z) = 1-z^d$ for some $d \le n^{1/5}$. Then $\prod_{j=J_1}^{J_2-1} |\wt{p}(h(e^{2\pi i \frac{j+\delta}{m}}))| \le \exp(Cn^{1/5}\log^5n)$ for any $\delta \in [0,1)$. 
\end{lemma}

\vspace{1mm}

By adapating the proof of the above lemma, we prove the following.

\vspace{1mm}

\begin{lemma}\label{uproduct}
Suppose $u(z) = z-\zeta$ for some $\zeta \in \partial \mathbb{D}$. Then, for any $\delta \in [0,1)$, we have $\prod_{j=J_1}^{J_2-1} |u(h(e^{2\pi i \frac{j+\delta}{m}}))| \le \exp(Cn^{1/5}\log^5n)$. 
\end{lemma}

\begin{proof}
First note that \begin{equation}\label{note} |u(h(e^{2\pi i \theta}))| \ge 1-|h(e^{2\pi i \theta})| \ge a^{10}. \end{equation} Define $g(t) = 2\log|u(h(e^{2\pi i (t+\frac{\delta}{m})}))|$. For notational ease, we assume $\delta = 0$; the argument about to come works for all $\delta \in [0,1)$. Since \eqref{note} implies $g$ is $C^1$, by the mean value theorem we have \begin{align}\label{mvt}\left|\frac{1}{m}\sum_{j=J_1}^{J_2-1} g\left(\frac{j}{m}\right)-\int_{J_1/m}^{J_2/m} g(t)dt\right| &= \left|\sum_{j=J_1}^{J_2-1} \int_{j/m}^{(j+1)/m} \left(g(t)-g\left(\frac{j}{m}\right)\right)dt\right| \nonumber\\ &\le \sum_{j=J_1}^{J_2-1} \int_{j/m}^{(j+1)/m} \left(\max_{\frac{j}{m} \le y \le \frac{j+1}{m}} |g'(y)|\right)\frac{1}{m} dt \nonumber\\ &\le \frac{1}{m^2}\sum_{j=J_1}^{J_2-1}\max_{\frac{j}{m} \le y \le \frac{j+1}{m}} |g'(y)|. \end{align} Since $w \mapsto \log|u(h(w))|$ is harmonic and $\log|u(h(0))| = \log|u(0)| = 0$, we have $$\int_0^1 g(t)dt = 2\int_0^1 \log |u(h(e^{2\pi i t}))|dt = 0,$$ and therefore
\begin{equation} \label{harmonicapplication}
\left|\int_{J_1/m}^{J_2/m} g(t)dt\right| \le \left|\int_{0}^{J_1/m}g(t)dt\right|+\left|\int_{J_2/m}^1 g(t)dt\right|.
\end{equation}
Since $$a^{10} \le \left|u(h(e^{2\pi i t}))\right| \le 2$$ for each $t$, we have
\begin{equation}\label{integralbounds}
\left|\int_{0}^{J_1/m}g(t)dt\right|+\left|\int_{J_2/m}^1 g(t)dt\right| \le 20\left(\frac{J_1}{m}+(1-\frac{J_2}{m})\right)\log n \le C\frac{\log^5 n}{n^{1/5}}.
\end{equation}
By \eqref{mvt}, \eqref{harmonicapplication}, and \eqref{integralbounds}, we have $$\left|\frac{1}{m}\sum_{j=J_1}^{J_2-1} g(\frac{j}{m})\right| \le C\frac{\log^5 n}{n^{1/5}}+\frac{1}{m^2}\sum_{j=J_1}^{J_2-1} \max_{\frac{j}{m} \le t \le \frac{j+1}{m}} |g'(t)|.$$
Multiplying through by $m$, changing $C$ slightly, and exponentiating, we obtain 
\begin{equation}\label{partialproductbound}
\prod_{j=J_1}^{J_2-1} \left|u(h(e^{2\pi i \frac{j}{m}}))\right|^2 \le \exp\left(Cn^{1/5}\log^5 n + \frac{1}{m}\sum_{j=J_1}^{J_2-1} \max_{\frac{j}{m} \le t \le \frac{j+1}{m}} |g'(t)|\right).
\end{equation}
Note $$g'(t_0) = \frac{\frac{\partial}{\partial t}\Big[|u(h(e^{2\pi i t}))|^2\Big] \Big|_{t=t_0}}{|u(h(e^{2\pi i t_0}))|^2}.$$ We first show \begin{equation}\label{derivativebound} \frac{\partial}{\partial t}\Big[|u(h(e^{2\pi i t}))|^2\Big] \Big|_{t=t_0} \le 500\end{equation} for each $t_0 \in [0,1]$. Let $\wt{d}_j = d_j$ for $j \le r_*$ and $\wt{d}_j = -d_j$ for $j > r_*$ so that $h(e^{2\pi i t}) = (1-a^{10})\sum_{j=1}^r \wt{d}_j e^{2\pi i tj}$. Then, $$\left|u\left(h(e^{2\pi i t})\right)\right|^2 = \left|(1-a^{10})\sum_{j=1}^r \wt{d}_j e^{2\pi i jt}-\zeta\right|^2$$ \begin{equation}\label{derivative} = (1-a^{10})^2\left|\sum_{j=1}^r \wt{d}_j e^{2\pi i jt}\right|^2-2\Real\left[(1-a^{10})\zeta \sum_{j=1}^r \wt{d}_j e^{2\pi i jt}\right]+1.\end{equation} The derivative of the first term is $$(1-a^{10})^2\sum_{j_1,j_2=1}^r \wt{d}_{j_1}\wt{d}_{j_2}2\pi(j_1-j_2)e^{2\pi i (j_1-j_2)t}.$$ Since $$\sum_{j=1}^r |\wt{d}_j| \le 4$$ and $$\sum_{j=1}^r j|\wt{d}_j| \le 4,$$ we get an upper bound of $250$ for the absolute value of the derivative of the first term of \eqref{derivative}. The derivative of the second term, if $\zeta = e^{i\theta}$, is $$2(1-a^{10})\sum_{j=1}^r \wt{d}_j \sin(2\pi j t+\theta)2\pi j,$$ which is also clearly upper bounded by (crudely) $250$. We've thus shown \eqref{derivativebound}. 

\vs

Recall $|u(h(e^{2\pi i \theta}))| \ge 1-|h(e^{2\pi i \theta})|$. For $j \in [J_1,J_2] \sub [C_6a^{1/2}m,(1-C_6a^{1/2})m]$, we use (by Lemma \ref{hproperties}) $$|h(e^{2\pi i \frac{j}{m}})| \le 1-c_5\frac{\min(\frac{j}{m},1-\frac{j}{m})}{\log^2 n}$$ to obtain $$\frac{1}{m}\sum_{j=J_1}^{J_2-1}\max_{\frac{j}{m} \le t \le \frac{j+1}{m}} |g'(t)| \le \frac{1}{m}\sum_{j=J_1}^{J_2-1} \frac{500}{\left(c_5\frac{\min(\frac{j}{m},1-\frac{j}{m})}{\log^2 n})\right)^2}.$$ Up to a factor of $2$, we may deal only with $j \in [J_1,\frac{m}{2}]$. Then we obtain \begin{align*}\frac{1}{m}\sum_{j=J_1}^{J_2-1}\max_{\frac{j}{m} \le t \le \frac{j+1}{m}} |g'(t)| &\le \frac{1}{m}\sum_{j=J_1}^{m/2} \frac{500m^2\log^4 n}{c_5^2j^2} \\ &\le \frac{500m\log^4 n}{c_5^2}\frac{2}{J_1} \\ &\le Cn^{1/5}. \end{align*}
\end{proof}

\vs

Let $\mc{Q}_n$ denote all polynomials of the form $(z-\alpha)(z-\beta)p(z)$ for $p \in \mc{P}_n$. 

\vspace{1mm}

\begin{corollary}\label{productbound}
For any $q \in \mc{Q}_n$ and $\delta \in [0,1)$, $\prod_{j\not \in \{0,m-1\}} |q(h(e^{2\pi i \frac{j+\delta}{m}}z))| \le \exp(Cn^{1/5}\log^5 n)$. 
\end{corollary}

\begin{proof}
Take $q \in \mc{Q}_n$; say $q(z) = (z-\alpha)(z-\beta)p(z)$ for $p \in \mc{P}_n$. For $j \in \{1,\dots,J_1-1\}$ and for $j \in \{J_2,\dots,m-2\}$, by Lemma \ref{unitdisk} we can bound $|q(h(e^{2\pi i \frac{j}{m}}z))| \le 4n$, to obtain \begin{equation}\label{smallj} \prod_{j \not \in \{J_1,\dots,J_2-1\}} |q(h(e^{2\pi i \frac{j+\delta}{m}}))| \le (4n)^{J_1-1+m-J_2-1} \le e^{Cn^{1/5}\log^5 n}.\end{equation} By applying Lemma \ref{uproduct} to $u(z) := z-\alpha$ and to $u(z) := z-\beta$ and multiplying the results, we see \begin{equation}\label{uterms} \prod_{j=J_1}^{J_2-1} |\overline{u}(h(e^{2\pi i \frac{j+\delta}{m}}))| \le e^{Cn^{1/5}\log^5 n},\end{equation} where $\overline{u}(z) := (z-\alpha)(z-\beta)$. Let $\wt{p}(z) \in \{1,1-z^d\}$ be the truncation of $p$ to terms of degree less than $n^{1/5}$. Then, since Lemma \ref{hproperties} gives $$|h(e^{2\pi i \frac{j+\delta}{m}})| \le 1-c_5\frac{\min\left(\frac{j}{m}+\delta,1-(\frac{j}{m}+\delta)\right)}{\log^2n} \le 1-c'n^{-1/5}\log^2n$$ for $j \in \{J_1,\dots,J_2-1\}$, we see \begin{equation}\label{approximation} \left|p\hspace{-.5mm}\left(h(e^{2\pi i \frac{j+\delta}{m}})\right)-\wt{p}\hspace{-.5mm}\left(h(e^{2\pi i \frac{j+\delta}{m}})\right)\right| \le ne^{-c'\log^2 n} \le e^{-c\log^2n}.\end{equation} Lemma \ref{tildepproduct} implies \begin{equation}\label{ptildebound} \prod_{j=J_1}^{J_2-1} |\wt{p}(h(e^{2\pi i \frac{j+\delta}{m}}))| \le e^{Cn^{1/5}\log^5 n}.\end{equation} By an easy argument given in \cite{chasesw}, \eqref{approximation} and \eqref{ptildebound} combine to give \begin{equation}\label{pterms} \prod_{j=J_1}^{J_2-1} |p(h(e^{2\pi i \frac{j+\delta}{m}}))| \le e^{C'n^{1/5}\log^5 n}.\end{equation} Combining $\eqref{smallj}, \eqref{uterms}$, and $\eqref{pterms}$, the proof is complete. 
\end{proof}

\vspace{1.5mm}

\begin{proposition}\label{eregionlowerbound}
For any $q \in \mc{Q}_n$, it holds that $\max_{w \in G_a} |q(w)| \ge \exp(-Cn^{1/5}\log^5 n)$. 
\end{proposition}

\begin{proof}
Let $g(z) = \prod_{j=0}^{m-1} q(h(e^{2\pi i \frac{j}{m}}z))$. For $z = e^{2\pi i \theta}$, with, without loss of generality, $\theta \in [0,\frac{1}{m})$, we have by Lemma \ref{hproperties} and Corollary \ref{productbound} $$|g(z)| \le \left(\max_{w \in G_a} |q(w)|\right)^2\prod_{j \not \in \{0,m-1\}} |q(h(e^{2\pi i (\frac{j}{m}+\theta)}))| \le \left(\max_{w \in G_a} |q(w)|\right)^2\exp(Cn^{1/5}\log^ 5n).$$ Thus, $\left(\max_{w \in G_a} |q(w)|\right)^2\exp(Cn^{1/5}\log^ 5n) \ge \max_{z \in \partial \mathbb{D}} |g(z)| \ge |g(0)| = 1$, where the last inequality used the maximum modulus principle (clearly $g$ is analytic). 
\end{proof}

\vspace{1.5mm}

The following lemma was proven in \cite{littlewoodcircle}. 

\vspace{1mm}

\begin{lemma}\label{hadamardregions}
Suppose $g$ is an analytic function in the open region bounded by $I_0$ and $I_a$, and suppose $g$ is continuous on the closed region between $I_0$ and $I_a$. Then, $$\max_{z \in I_{a/2}} |g(z)| \le \left(\max_{z \in I_0} |g(z)|\right)^{1/2}\left(\max_{z \in I_a} |g(z)|\right)^{1/2}.$$
\end{lemma}

\vs

\begin{proof}[Proof of Theorem \ref{arclowerbound}]
Take $f \in \mc{P}_n$, and let $g(z) = (z-\alpha)(z-\beta)f(z)$. A straightforward geometric argument yields $$|g(z)| \le \frac{|(z-\alpha)(z-\beta)|}{1-|z|} \le \frac{2}{\sin(a)} \le 3n^{2/5}$$ for $z \in I_0$. Letting $L = ||g||_{I_a}$, Lemma \ref{hadamardregions} then gives $$\max_{z \in I_{a/2}} |g(z)| \le (3Ln^{2/5})^{1/2}.$$ Since we then have $$\max_{z \in I_{a/2}\cup I_a} |g(z)| \le \max(L,(3Ln^{2/5})^{1/2}),$$ the maximum modulus principle implies $$\max_{z \in G_a} |g(z)| \le \max(L,(3Ln^{2/5})^{1/2}).$$ By Proposition \ref{eregionlowerbound}, we conclude $$\exp(-Cn^{1/5}\log^5 n) \le \max\left(L,(3Ln^{2/5})^{1/2}\right).$$ Thus, $$||f||_{I_a} \ge \frac{1}{4}||g||_{I_a} = \frac{L}{4} \ge \exp(-C'n^{1/5}\log^5 n),$$ as desired.
\end{proof}

\section{Appendix: Proof of Lemma \ref{unitdisk}}

We thank Fedor Nazarov for a simpler proof of Lemma \ref{unitdisk}, which we include below.

\begin{claim}\label{lemmaclaim}
Let $\mc{F}$ be a compact family of (uniformly) bounded real Lipschitz functions on $[0,1]$ such that $\int_0^{1/2} f < \int_{1/2}^1 f$ for every $f \in \mc{F}$. Then there exist $M, \ep > 0$ so that for all $m > M$, $m_* \in ((\frac{1}{2}-\ep)m, (\frac{1}{2}+\ep)m)$, and $f \in \mc{F}$, it holds that \begin{equation}\label{claimlemma1} \sum_{j=1}^{m_*} \frac{1}{\log^2(j+3)}f\left(\frac{j}{m}\right) < \sum_{j=m_*+1}^m \frac{1}{\log^2(j+3)}f\left(\frac{j}{m}\right).\end{equation}
\end{claim}

\begin{proof}
By compactness, there exists $\ep > 0$ so that for all $\gamma \in (\frac{1}{2}-\ep,\frac{1}{2}+\ep)$ and all $f \in \mc{F}$, we have \begin{equation}\label{intdiffep} \int_0^\gamma f(x)dx < \int_\gamma^1 f(x)dx-\ep.\end{equation} Quickly note, for $C > 0$ a uniform upper bound on $\max_{x \in [0,1]} |f(x)|$, $f \in \mc{F}$, we have \begin{align}\label{quicknotee} \frac{1}{m}\sum_{j=1}^m \left[\frac{1}{\log^2(j+3)}-\frac{1}{\log^2(m+3)}\right]\left|f\left(\frac{j}{m}\right)\right| &\le C\frac{1}{m}\left[\sum_{j=1}^{\frac{m}{\log^3(m+3)}} 1+\sum_{j=\frac{m}{\log^3(m+3)}}^m \frac{\log\log(m+3)}{\log^3(m+3)}\right] \\ \nonumber &\le 2C\frac{\log\log(m+3)}{\log^3(m+3)} \\ \nonumber &= o(\frac{1}{\log^2(m+3)}) \end{align} as $m \to \infty$. As \eqref{claimlemma1} is equivalent to $$\frac{\log^2(m+3)}{m}\sum_{j=1}^{m_*} \frac{1}{\log^2(j+3)}f\left(\frac{j}{m}\right) < \frac{\log^2(m+3)}{m}\sum_{j=m_*+1}^m \frac{1}{\log^2(j+3)}f\left(\frac{j}{m}\right),$$ by \eqref{quicknotee} it suffices to prove \begin{equation}\label{heya} \frac{1}{m}\sum_{j=1}^{m_*} f\left(\frac{j}{m}\right) < \frac{1}{m}\sum_{j=m_*+1}^m f\left(\frac{j}{m}\right)-\frac{\ep}{2},\end{equation} say (for $m$ large enough and $m_* \in ((\frac{1}{2}-\ep)m,(\frac{1}{2}+\ep),m)$). But the LHS becomes arbitrarily close to $\int_0^{m_*/m} f(x)dx$, and the RHS becomes arbitrarily close to $\int_{m_*/m}^1 f(x)dx-\frac{\ep}{2}$, so \eqref{heya} is established by \eqref{intdiffep}. 
\end{proof}

Now, letting $f(x) = \frac{1}{2}-\frac{1}{2}\left(\frac{\sin(x/2)}{x/2}\right)^2$ for $x \in (0,1]$ and $f(0) = 0$, and then setting $f_c(x) = c^{-4}f(cx)$ for $c > 0$ and $x \in [0,1]$ and $f_0(x) = \frac{x^4}{24}$, we will apply Claim \ref{lemmaclaim} to the family $\mc{F} := \{f_c : c \in [0,C]\}$, for a suitable absolute $C > 0$. An easy computation shows that $\mc{F}$ is indeed a compact family of bounded Lipschitz functions. The condition that $\int_0^{1/2} f_c < \int_{1/2}^1 f_c$ for all $c \in [0,C]$ is equivalent to $\int_0^a f(x)dx < \int_a^{2a} f(x)$ for all $a > 0$, which is equivalent to $$\int_0^b \left(\frac{\sin x}{x}\right)^2dx > \int_b^{2b} \left(\frac{\sin x}{x}\right)^2dx$$ for all $b > 0$, which is easily verified\footnote{As $\frac{\sin x}{x}$ decreases on $[0,\pi]$, the case $b \le \frac{\pi}{2}$ is immediate. For $b > \frac{\pi}{2}$, we can do $\int_b^{2b} (\frac{\sin x}{x})^2dx < \int_{\pi/2}^\infty \frac{1}{x^2}dx = \frac{2}{\pi}$, which suffices since, by monotonicity, $\int_0^b (\frac{\sin x}{x})^2 dx > \int_0^{\pi/2} (\frac{\sin x}{x})^2 dx \ge \frac{\pi}{2}(\frac{2}{\pi})^2 = \frac{2}{\pi}$.}.

\begin{proof}[Proof of Lemma \ref{unitdisk}]
The proof of Lemma \ref{hproperties} shows that $\tilde{h}(e^{it}) \in \ol{D}$ if $t \in [-\pi,\pi]\setminus[-\frac{1}{100},\frac{1}{100}]$, say. So we may assume $|t| \le \frac{1}{100}$. First note that \begin{align}\label{jerry} \left|\Imag[\wt{h}(e^{it})]\right| &= \wt{\lam}_a\sum_{j=1}^r \ep_j d_j\sin(jt) \\ \nonumber &\le \wt{\lam}_a\sum_{j=1}^r d_j j |t| \\ \nonumber &\le 2|t|. \end{align} Also, \begin{align}\label{tom} \Real[\wt{h}(e^{it})] &= \wt{\lam}_a\sum_{j=1}^r \ep_j d_j\cos(jt) \\ \nonumber &\ge \wt{\lam}_a \sum_{j=1}^r \ep_j d_j \left(1-\frac{j^2t^2}{2}\right) \\ \nonumber &= 1-\frac{1}{2}t^2\wt{\lam}_a\sum_{j=1}^r \ep_j j^2 d_j \\ \nonumber &\ge 1-\frac{1}{2}t^2\wt{\lam}_a\cdot 21 \\ \nonumber &> 0.\end{align} Finally, using the identity $$\frac{\cos x-1+\frac{x^2}{2}}{x^2} = \frac{1}{2}-\frac{1}{2}\left(\frac{\sin(x/2)}{x/2}\right)^2,$$ we see that $$\Real[\wt{h}(e^{it})] = \wt{\lam}_a\left[\sum_{j=1}^{r_*} \frac{1}{\log^2(j+3)}\left(\frac{1}{j^2}-\frac{t^2}{2}\right)-\sum_{j=r_*+1}^r \frac{1}{\log^2(j+3)}\left(\frac{1}{j^2}-\frac{t^2}{2}\right)\right]$$ $$\hspace{15mm} +\wt{\lam}_ar^4t^6\left[\sum_{j=1}^{r_*} \frac{1}{\log^2(j+3)}f_{t r}(\frac{j}{r})-\sum_{j=r_*+1}^r \frac{1}{\log^2(j+3)}f_{t r}(\frac{j}{r})\right].$$ By Claim \ref{lemmaclaim}, we then see $$\Real[\wt{h}(e^{it})] \le \wt{\lam}_a\left[\sum_{j=1}^{r_*} \frac{1}{\log^2(j+3)}\left(\frac{1}{j^2}-\frac{t^2}{2}\right)-\sum_{j=r_*+1}^r \frac{1}{\log^2(j+3)}\left(\frac{1}{j^2}-\frac{t^2}{2}\right)\right],$$ which is at most $1-10t^2$ by our choice of $r_*$. Combining with \eqref{tom} and \eqref{jerry}, we see \begin{align*}\left|\wt{h}(e^{it})\right|^2 &= \left(\Real[\wt{h}(e^{it})]\right)^2 + \left(\Imag[\wt{h}(e^{it})]\right)^2 \\ &\le (1-10t^2)^2+4t^2 \\ &\le 1-6t^2 \\ & \le 1, \end{align*} as desired.
\end{proof}

\vs

\section{Acknowledgments} 

I would like to thank Shyam Narayanan for providing an extension to all $q \in (0,1)$, and Fedor Nazarov for an easier proof of Lemma \ref{unitdisk}.

\end{document}